\documentclass[twoside]{article}

\usepackage[accepted]{aistats2023}
\usepackage{packagesList}
\theoremstyle{plain}
\newtheorem{theorem}{Theorem}[section]
\newtheorem{lemma}{Lemma}[section]
\newtheorem{proposition}{Proposition}[section]

\newtheorem{fact}{Fact}[section]

\newtheorem*{discussion*}{Discussion}

\theoremstyle{definition}
\newtheorem{definition}{Definition}[section]
\newtheorem{remark}{Remark}[section]

\newcommand{\cond}{\textbf{C}\!\!}

\newtheorem{assumption}{Assumption}

\newtheorem*{remark*}{Remark}
\newtheorem{corollary}{Corollary}[section]

\numberwithin{equation}{section}

\newcommand{\nset}{\mathbb{N}}
\newcommand{\rset}{\mathbb{R}}

\newcommand{\ind}{\mathds{1}}
\newcommand{\abs}[1]{\lvert{#1}\rvert}

\newcommand{\ffrac}[2]{\ensuremath{\frac{\displaystyle #1}{\displaystyle #2}}}

\newcommand{\un}[1]{\ind{\left\{#1\right\}}}
\newcommand{\indic}{\un} 
\newcommand{\PP}[1][]{\ifthenelse{\equal{#1}{}}{\ensuremath{\mathbb{P}}}{\ensuremath{\mathbb{P}\left( #1 \right) }}}
\newcommand{\EE}[1][]{\ifthenelse{\equal{#1}{}}{\ensuremath{\mathbb E}}{\ensuremath{{\mathbb E}\left[ #1 \right]}}}
\newcommand{\Var}[1][]{\ifthenelse{\equal{#1}{}}{\ensuremath{\matrhm{Var}}}{\ensuremath{{\mathrm{Var}}\left[ #1 \right]}}}
\newcommand{\Cov}[1][]{\ifthenelse{\equal{#1}{}}{\ensuremath{\mathrm{Cov}}}{\ensuremath{{\mathrm{Cov}}\left[ #1 \right]}}}

\newcommand{\rCV}{\widehat{\mathcal{R}}_{CV}}
\newcommand{\rCVC}{\widehat{\mathcal{R}}^{corr}_{CV}}

\DeclareMathOperator*{\argmin}{arg\,min}

\usepackage{xspace}

\newcommand\ie{\emph{i.e.}\xspace}

\newtheorem{prop}{Proposition}[section]

\usepackage{xcolor}

\newcommand{\CV}{\textsc{CV}\xspace}

\newcommand{\loo}{\emph{l.o.o.}\xspace}

\newcommand\eg{\emph{e.g. }\xspace}

\newcommand{\TR}{\mathcal{R}}
\newcommand{\ER}[1][]{\widehat{ \mathcal{R}}}
\newcommand{\risk}{\mathcal{R}}

\newcommand{\data}{\mathcal{D}}
\newcommand{\DD}{\data}

\newcommand{\BiasCV}{\mathrm{Bias}}
\newcommand{\DevCV}{D_{CV}}
\newcommand{\DevT}{D_{\alg[T_{1:K}]}}
\newcommand{\DevAll}{D_{\alg[n]}}

\newcommand{\Kfold}{\textrm{K-fold}}

\newcommand{\alg}{\mathcal{A}}
\newcommand{\algdd}{\mathcal{A}_{\DD}}

\newcommand{\loss}{\ell}
\DeclareMathOperator{\sign}{sign}

\usepackage[round]{natbib}

\date{November 2021}

\begin{document}
\twocolumn[

\aistatstitle{On the bias of K-fold cross validation with stable learners}

\aistatsauthor{ Anass Aghbalou \And François Portier  \And Anne Sabourin   }

\aistatsaddress{ LTCI, Télécom Paris\\ 
	Université Paris-Saclay \And   CREST\\Ensai \And CNRS, MAP5\\Université Paris Cité } ] 

\begin{abstract}
	This paper investigates the efficiency of the K-fold cross-validation (CV) procedure and a debiased version thereof as a means of estimating the generalization risk of a learning algorithm. We work  under the  general assumption of uniform algorithmic stability. We show that the K-fold risk estimate may not be consistent under such general stability assumptions,  by constructing non vanishing lower bounds on the error in realistic contexts such as regularized empirical risk minimisation and stochastic gradient descent. We thus advocate the use of a debiased version of the K-fold and prove an error bound with  exponential tail decay regarding this version. Our result is applicable to the large class of uniformly stable algorithms, contrarily to earlier works focusing on specific tasks such as density estimation. We illustrate the relevance of the debiased K-fold CV on a simple model selection problem and demonstrate empirically the usefulness of the promoted approach on real world classification and regression datasets.


\end{abstract}

\section{INTRODUCTION}
Introduced in \cite{stone74},  cross-validation (\CV) is a popular tool in statistics for estimating the generalization risk of a learning algorithm. It is also the mainstream approach for model and  parameter selection. Despite its widespread use, it has been shown in several contexts that \CV schemes fail to select the correct model unless the test fraction is negligible in front of the sample size. Unfortunately, this excludes the widely used  \Kfold\:\CV. This suboptimality has been pinned in the linear regression framework by \cite{BURMAN89,shao97,Yuhong2007}, then in other  specific frameworks such as  density estimation \citep{arlot2008} and classification \citep{yang2006}. 
The theoretical properties of \CV\ procedures for model selection in wider settings are
 notoriously difficult to establish and remain the subject of active research \citep{bayle2020cross,wager2020cross}. 

To tackle the suboptimality of \Kfold,
\cite{BURMAN89,fushiki2011estimation} have proposed to add some
debiasing  correction terms to the \Kfold\xspace\CV estimate in order  to
improve the convergence rate. However, the analysis conducted in these
works is purely asymptotic and focuses only on ordinary linear
regression. More recently \cite{arlot2016VFchoice} conduct a non
asymptotic study for the bias corrected \Kfold\xspace in the density
estimation framework and show the benefits of such a
correction. Nonetheless, the latter study relies on closed-form
formulas for risk estimates which are valid only for  histogram rules.
To summarize, the statistical consistency of the debiased version has been established only in  specific frameworks pertaining to classical statistics. In addition, to our best knowledge, the consistency of \Kfold\xspace without correction has not been proved nor disproved in the existing literature. 

The main purpose of this paper is to establish non-asymptotic results (upper and lower bounds) regarding the error of the \Kfold\xspace risk estimate, under realistic  assumptions which are valid for a
wide class of modern algorithms (regularized empirical risk
minimization, neural networks, bagging, SGD, \emph{etc}\dots), namely an \emph{algorithmic stability}  assumption discussed below. In other words, the question we seek to answer is as follows: 
\begin{itemize}
	\item Is \Kfold\xspace cross-validation consistent  under  algorithmic stability assumptions? If not, how about the bias corrected \Kfold?	
\end{itemize}

The notion of \emph{algorithmic stability} and its consequences in learning theory has received much attention since its introduction in \cite{DEvroy-79}. This property allows to obtain generalization bounds for a large class of a learning algorithms such as  k-nearest-neighbors \citep{DEvroy-79}, empirical risk minimizers \citep{kearns1999algorithmic}, regularization networks \citep{BousquetElisseeff-2000}, bagging \citep{elisseeff05a} to name but a few. For an exhaustive review of the different notions of \emph{stability} and their consequences on the generalization risk of a learning algorithms, the reader is referred to \cite{kutin2002}. Our  working assumption  in this paper is  \emph{uniform stability}, which   encompasses many algorithms such as  Support Vector Machine  \citep{bousquet2002stability}, regularized empirical risk minimization \citep{Zhang2004,wibisono2009}, stochastic gradient descent \citep{hardt16} and   neural networks with a simple architecture \citep{charles18a}.

\paragraph{\bf Related Work on \Kfold\xspace\CV\ with Stable Learners.}
In \cite{kale2011cross,kumar2013near}, \Kfold\ \CV\ for risk estimation  is envisioned under stability assumptions regarding the algorithm. It is shown that the \Kfold\xspace risk estimate has a much smaller variance than the simple hold-out estimate and the amount of variance reduction is quantified.  Another related work is \cite{abou2017} who builds upon a variant of algorithmic stability, namely $L^q$-\emph{stability} to derive PAC upper bounds for \Kfold\ \CV error estimates. Other  results regarding the  asymptotic behavior  of  \Kfold\ estimates can be found in \cite{austern2020,bayle2020cross}.

However, none of the results mentioned above imply a universal upper
bound regarding the \Kfold\ neither for risk estimation nor for model
selection. Indeed their focus is on the the \emph{variance term} of
the \Kfold\ error, while they do not take into account the high \emph{bias}
generally  induced by this \CV\ scheme (see
\cite{shao97,arlot2016VFchoice} for instance).
To our best knowledge, the literature on algorithmic
stability is silent about the consistency of \Kfold\:CV\,-- the most
widely used \CV\ scheme -- in a generic stability setting. Filling
this gap is the main purpose of the present paper.

\paragraph{Contributions and Outline.}
We introduce the necessary background and notations about \CV\ risk estimation and algorithmic stability in Section~\ref{sec:backround}.
Section~\ref{sec:CV-UB} is intended to give some context about provable guarantees regarding \Kfold\ CV\ scheme, namely 
we  state and prove a generic upper bound on the error of the generalization risk estimate for uniformly stable algorithms. However, with realistic stability constants,  the obtained upper bound is not satisfactory, in so far as  it does not vanish as the sample size $n\rightarrow \infty$.  


Our main  contributions are gathered in sections~\ref{sec:KF-LB} to \ref{sec:model-selec} and  may be summarized as follows:
\begin{enumerate}
\item One may wonder whether the looseness of the bound for the  \Kfold\ \CV error  is just an artifact from our proof. 
  We answer in the negative by deriving a lower bound on the \Kfold\ error \:(Section~\ref{sec:KF-LB}) in two different contexts, specifically, regularized empirical risk minimization and  stochastic gradient optimization. The latter bound shows that under the uniform stability assumption alone, \Kfold\ \CV\ is inefficient in so far as it can fail in estimating the generalization risk of a uniformly stable algorithm.
\item 
  We analyze  a corrected \Kfold\ procedure and prove a  PAC generalization upper bound covering the general case of uniformly stable learners. As a consequence, the corrected version of the \Kfold\ is shown to be  efficient in contrast to the standard version. The   corrected \Kfold\ scheme  has been investigated in~\cite{BURMAN89,burman1990estimation,fushiki2011estimation,arlot2016VFchoice} in the particular frameworks of  ordinary linear regression and density estimation. Furthermore, the analysis in the latter references relies on strong regularity assumptions (further details are given in Section \ref{sec:BKF-UB}) which  aren't satisfied by many modern learning rules like Support Vector Machine (SVM), stochastic gradient descent methods, bagging, etc.
  Instead our upper bound covers the general case of uniformly stable learners. 
  As an example of application, we show that the debiased \Kfold\:permits to select a model within a finite collection in a  risk consistent manner (Section \ref{sec:model-selec}). In other words, the excess risk of the selected model tends to $0$ as $n\rightarrow \infty$. Finally we demonstrate empirically the added value of the debiased \Kfold\:compared with the standard one in terms of the test error of the selected model.
        \end{enumerate}



\section{BACKGROUND, NOTATIONS AND WORKING ASSUMPTIONS}
\label{sec:backround}
\subsection{Notations}
We place ourselves in the following general learning setting.  One  receives a collection of independent and identically distributed random vectors $\DD=(O_1,\ldots,  O_n)$ 
lying in a sample space $\mathcal Z$, with common distribution $P$. For any $n\in \nset$, let $[n]$ denote the set of integers $\{1,2,\dots,n\}$. Consider a class of predictors $\mathcal{G}$ and a loss function  
$\loss: \mathcal G \times \mathcal Z \to \rset$, so that $\loss(g,O)$
be the error of $g$ on the observation $O\in \mathcal Z$. As an example, in
the supervised learning setting
$\mathcal{Z} = \mathcal{X}\times \mathcal{Y}$,  $g$ is a  mapping
$\mathcal{X}\to \mathcal{Y}$ and for $o = (x,y)$ the loss function writes as 
$\ell(g,o) = \ell(g(x), y)$. However our results are not limited to
the supervised setting.  Given a subsample
$\DD_T =\{O_i \mid i \in T \}$ indexed by $T\subset [n]$ and an
algorithm (or learning rule) $\alg$, we denote by
$\alg(T) \in \mathcal G$ the predictor obtained by training $\alg$ on
$\DD_T$. 
   We consider in this paper deterministic algorithms, that is,  given a subsample $\DD_T$, the output of the algorithm $\alg(T)$ is non random. 
  We thereby neglect the randomness brought \emph{e.g.} by optimization routines. The case of random algorithms can be covered  at the price of additional notational burden. For the sake of readability we restrict ourselves to  deterministic algorithms  in the main paper and show how to relax it in the supplementary material (Section~\ref{sec:randomized}), in order to cover the case of random algorithms such as stochastic gradient descent (\textrm{SGD}) or neural networks. 
 This extension is in particular necessary to one of our counter-examples (Section~\ref{theo:K-fold-LB-SGD}) where we build a lower bound for the \textrm{SGD} algorithm.

The generalization risk of the predictor $\alg(T)$ is then
\[
\TR\big[\alg(T)\big]=\EE\left[ \loss\left(\alg\left(T\right),O\right)\mid \DD_T \right],
\]
where $O$ is independent from $\DD_T$. Notice that the randomness in the latter  expectation stems from  the novel observation $O$ only while the trained algorithm $\alg(T)$ is fixed. 
The quantity of interest here is the generalization risk of the learning rule trained on the full dataset, $\TR\big[\alg([n])\big]$. 
The hold-out estimate of the latter involves  a validation index set  $V$ disjoint from $T$ and writes as 
\[\ER\big[\alg(T), V\big]=\frac{1}{ n_V} \sum_{i\in V} \loss(\alg(T),O_i),
\]
where 
$n_V=card(V)$. \\
Given a family of validation sets in  $[n]$,  $V_{1:K} = (V_j)_{j=1,\ldots, K}$, 
the \Kfold\xspace\CV  estimator of the generalization risk of $\alg([n])$ is
\begin{equation}\label{def:risk-CV}
\rCV\left[\alg,V_{1 : K}\right]=
\frac{1}{ K} \sum_{j= 1}^{K}\ER\big[\alg({T_j}) , V_j\big],
\end{equation}
where $T_j = [n] \backslash V_j$. 
For clarity reasons, we suppose further that $n$ is divisible by $K$ so that $n/K$ is an integer. This condition guarantees, that all validation sets have the same cardinal $n_{V}=n/K$.

\subsection{Algorithmic Stability}\label{sec:algoStability}
An algorithm $\alg$ is called stable if removing a training point $O_i$ from $\DD_T$ ($i\in T$) or replacing $O_i$ with an independent observation $O'$ drawn from the same distribution does not change much the risk of the output. Formally, for $i\in T \subset [n]$ as above, let $T^{\setminus i }= T\setminus \{i\}$, so that $\alg(T^{\setminus i})$ is the output of $\alg$ trained on $\DD_T\setminus\{O_i\}$. 
Denote similarly  $\alg(T^i)$ the output of $\alg$ trained on $\DD_T\setminus\{O_i\} \cup\{O'\}$. 
The notion of \emph{hypothesis stability} was first introduced in \cite{DEvroy-79} to derive non asymptotic guarantees for the leave-one-out \CV (\loo). In this paper, we consider instead   \emph{uniform stability}, an assumption used in \cite{bousquet2002stability,wibisono2009,hardt16,feldman2019high} to derive probability upper bounds for the training error and \loo estimates. With the above notations, uniform stability is defined as follows.

\begin{definition}\label{def:unif-stable}
	An algorithm $\alg$ is said to be   $(\beta_t)_{1\leq t\leq n}$ uniformly stable with respect to a loss function $\loss$ if, for any $T\subset [n]$, $i\in T$ it holds  that  
	\begin{equation}
	\left|\loss\left(\alg(T),O\right)-
	\loss\left(\alg(T^{\backslash i}),O\right)\right| \leq \beta_{n_T}, 
	\end{equation}
	with $P$-probability one.
      \end{definition}

        Many widely used Machine Learning algorithms are uniformly stable in the sense of Definition~\ref{def:unif-stable}. In particular $\beta_n\leq \ffrac{C}{n}$  for SVM and least square regression with the usual mean squared error, and for SVM classification with the soft margin loss \citep{bousquet2002stability}.  Up to a minor definition of uniform stability accounting for randomness (see Section~\ref{sec:randomized} in the supplement), many extensively used stochastic gradient methods are also uniformly stable, such as \emph{e.g.} SGD  with convex and non convex losses \citep{hardt16} or  RGD (randomized coordinate descent) and  SVRG (stochastic variance reduced gradient method) with loss functions verifying the Polyak-Łojasiewicz condition \citep{charles18a}.

The following simple fact concerns the effect of removing $n'$ training points on  uniformly stable algorithms. 
\begin{fact}
	\label{fact:Bias-UB}
	Let $\alg$ be a decision rule which is $(\beta_t)_{1\leq t  \leq n}$ uniformly stable, additionally suppose that the sequence $(\beta_t)_{1\leq t  \leq n}$ is decreasing, then for any $T\subset [n]$,  
        one has, 
	\[\left|\loss\left(\alg\left(\left[n\right]\right),O\right)-
		\loss\big(\alg(T),O\big)\right| \leq \sum_{i=n_T+1}^{n}\beta_i. 
	\]
\end{fact}
\begin{remark}
	Definition \ref{def:unif-stable} and Fact \ref{fact:Bias-UB}  play a key role in our proofs. Namely we use Fact~\ref{fact:Bias-UB} to control the bias of the CV risk estimate and Definition~\ref{def:unif-stable} to derive a probability upper bound on its deviations via McDiarmid's inequality.
\end{remark}

We also rely on the fact that the  training and validation sets of \Kfold\ \CV verify a certain balance condition.
\begin{fact}\label{fact:mask-property}
	For the \Kfold\ \CV the validation sets $V_1,V_2,\dots V_K$  satisfies
	\begin{equation}\label{"card-condition"}
	card(V_j)=n_V\quad \forall j \in \llbracket1,K\rrbracket,
	\end{equation}
	for some $n_V \in \llbracket1,n\rrbracket$.\,Moreover it holds that 
	\begin{equation}\label{"key-condition"}
	\frac{1}{K}\sum_{j=1}^{K} {\indic{l \in V_j} }=\frac{n_V}{n} \quad \forall l \in [n]. 
	\end{equation}
\end{fact}
Because $T_j=[n]\backslash V_j$,  if  \eqref{"key-condition"} holds, then the training sets $T_j$ verify a similar equation, that is, 
$$	\frac{1}{K}\sum_{j=1}^{K}  {\indic{l \in T_j} }=\frac{n_{T}}{n} \quad \forall l \in [n].$$
We prove Fact~\ref{fact:mask-property} in the supplement (Lemma~\ref{lemma:CV-scheme-property}). 


Throughout this paper we work under the following uniform stability assumption combined with a boundedness assumption regarding the cost function.
\begin{assumption}[Stable algorithm]
	\label{assum:stability-setting}
	The algorithm $\alg$ is $(\beta_t)_{1\leq t\leq n}$ uniformly stable with respect to a cost function $\loss$ that satisfies 
        $$   \forall O \in \mathcal{Z} \:,\:
        \forall T \subset [n] ,\: \left|\loss\left(\alg\left(T\right), O \right)\right| \leq L.$$
	
      \end{assumption}
      


\section{UPPER BOUNDS FOR \textrm{K-FOLD} RISK ESTIMATION}
\label{sec:CV-UB}
Our first  result Theorem~\ref{theo:CV-stable-bound} is a generic upper bound on the error of the generalization risk estimate for stable algorithms satisfying Assumption~\ref{assum:stability-setting}. 
 Our upper bound is of the same order of
magnitude as existing results in the literature which apply to other contexts, 
\emph{e.g.} in \cite{bousquet2002stability} for \loo \CV or
in \cite{cornec17} for the specific case empirical risk minimizers, although our techniques of proof are different.
The fact that the upper bound does not  vanish with large sample sizes $n$ with realistic stability constants (Corollary~\ref{coro:cv-stable-UB}) gives some context and motivates the rest of this work. 
 


 \begin{theorem}\label{theo:CV-stable-bound}
 Consider a stable learning algorithm $\alg$ satisfying  Assumption~\ref{assum:stability-setting}. Then, we have with probability $1-2\delta$,
 \begin{align*}
 	\bigg|\rCV\left[\alg,V_{1:K}\right]-\TR\big[&\alg([n])\big]\bigg|\leq \sum_{i=n_T+1}^{n}\beta_i\\&+(4\beta_{n_T}n_T+2L)\sqrt{\ffrac{\log(1/\delta)}{2n}}.
 \end{align*}

Where $L$ is the upper bound on the cost function $\loss$ from Assumption~\ref{assum:stability-setting}.	
\end{theorem}

\begin{proof}[Sketch of proof]
	Define the average risk of the family $\big(\alg(\DD_{T_j})\big)_{1\leq j \leq K}$
	\begin{equation}\label{eq:true-rcv-def}
		\risk_{CV}\left[\alg,V_{1:K}\right]= \frac{1}{K}\sum_{j=1}^{K}\risk\big[\alg(T_j)\big],
	\end{equation}
	then write the following decomposition 
	\begin{equation}\label{ineq:error-decomp}
		\rCV\left[\alg,V_{1:K}\right]-\TR\big[\alg([n])\big] = \DevCV+\BiasCV,
	\end{equation}
	with \begin{align}
	&\DevCV= \rCV\left[\alg,V_{1:K}\right]-  \risk_{CV}\left[\alg,V_{1:K}\right]\label{def:dev_CV},  \\
	&\BiasCV= \mathcal{R}_{CV}\left[\alg,V_{1:K}\right]-\TR\big[\alg([n])\big]\label{def:bias_CV}.
	\end{align}
	The proof consists in bounding each term of the above decomposition independently. The term $\DevCV$ measure the deviations of $\rCV$ from its mean and it can be controlled using McDiarmid's inequality (Proposition \ref{prop:Mcdiarmids-ineq}). The second term $\BiasCV$ is  controlled using Fact~\ref{fact:Bias-UB}. The detailed proof is deferred to the appendix.
\end{proof}
As discussed in the background section~\ref{sec:algoStability}, typical uniform stability constants $\beta_n$ for standard  algorithms satisfy  $\beta_n\leq \frac{C}{n}$. In this case, Theorem~\ref{theo:CV-stable-bound} yields the following corollary. 

\begin{corollary}\label{coro:cv-stable-UB}
  Consider a stable learning algorithm $\alg$ satisfying  Assumption~\ref{assum:stability-setting}  with stability parameter $\beta_n\leq \frac{C}{n} $. Then, we have with probability $1-2\delta$,
  \begin{align*}         
    \bigg|\rCV\left[\alg,V_{1:K}\right]-\TR\big[
    &\alg([n])\big]\bigg|\leq C \log\left(\ffrac{K}{K-1}\right) \\  
    &+(4C+2L)\sqrt{\ffrac{\log(1/\delta)}{2n}},
  \end{align*}

  wich does not converge  to $0$ as $n\to \infty$.
\end{corollary} 

\begin{proof}
Recall  that $\ffrac{n}{n_T}=\frac{K}{K-1}$ and write 
$$\sum_{i=n_T+1}^{n}\beta_i \le
C\sum_{i=n_T+1}^{n}\frac{1}{i}\leq
C \log\left(\frac{K}{K-1}\right). $$
       The result follows from
Theorem~\ref{theo:CV-stable-bound}.
\end{proof}

\begin{remark}[Bias and Variance of \Kfold\ \CV]\label{remark:bias-variance-stability}
The two terms  of the sum in  the upper bound of  Theorem~\ref{theo:CV-stable-bound} correspond respectively to a  variance  and a bias  term, namely~$\DevCV$ and ~$\BiasCV$ in the error decomposition~\eqref{ineq:error-decomp}.\\
On the one hand, the term $(4\beta_{n_T}n_T+2L)\sqrt{\ffrac{\log(1/\delta)}{2n}}$ reflects the variance of the \CV\ procedure. When $\beta_n\leq \frac{C}{n}$ it yields the usual 
rate $1/\sqrt{n}$.  
On the other hand, the term $\sum_{i=n_T+1}^{n}\beta_i$ reflects the
bias of \Kfold\ \CV. Contrarily to the variance term, it does not
vanish as $n\to \infty$, even when $\beta_n\le C/n$. Finally, Notice that, as the number of folds $K$ increases,  the training size $n_T$ gets closer to  the sample size $n$  and the bias of \Kfold\,vanishes. However, for computational efficiency, increasing the number of folds is not always desirable. 

\end{remark}

One may wonder whether the looseness of the bias term $\sum_{i=n_T+1}^{n}\beta_i$  is just an artifact from our proof. In the next theorem we answer in the negative by deriving a lower bound for the \Kfold. The latter bound shows that under the uniform stability assumption alone, \Kfold\ \CV\ is inefficient in so far as it can fail in estimating the generalization risk of a uniformly stable algorithm.

\section{LOWER BOUND FOR THE \textrm{K-FOLD}\ ERROR UNDER ALGORITHMIC STABILITY}
\label{sec:KF-LB}
To construct lower bounds on \Kfold\ \CV error, we consider two families of algorithms that   satisfy  the  uniform stability  hypothesis with parameter $\beta_n$ scaling as $1/n$. Namely, regularized empirical risk minimizers (\textrm{RERM}) and stochastic gradient descent (\textrm{SGD}).
\subsection{Regularized Empirical Risk Minimization}
The  first counter-example that we build to prove a lower bound is formulated in a regression framework. In particular, we consider $L^2$-regularized empirical risk minimization ($L^2$-\textrm{RERM}) 
\begin{equation}\label{def:RERM}
	\alg= \argmin_{g\in \mathcal{G}} \left\{\frac{1}{n}\sum_{i=1}^{n}\loss\left(g,O_i\right)+\lambda_n\left\| g\right\|^2_{\mathcal{G}}\right\}, 
\end{equation}
	
where $\loss$ is a convex loss function and $\mathcal{G}$ is a hypothesis space.\\  Under some mild assumptions on $\loss$ and the input space $\mathcal{Z}$, an  $L^2$-RERM  algorithm is uniformly stable with $\beta_n\leq \frac{C}{n}$   (see \eg\cite{Zhang2004,wibisono2009,liu2017} for further details).\\
The next result confirms that uniform stability alone is not sufficient to ensure the consistency of  \Kfold\ \CV for RERM.

\begin{theorem}[Non vanishing lower bound on the \Kfold\ error]\label{theo:K-fold-LB}  Consider the regression problem for a random pair
$O = (X,Y) \in \mathcal{Z} = \mathcal{X}\times\mathcal{Y}$,  
loss function $\loss(g , o)=(g(x) -y)^2$ and hypothesis class $\mathcal{G}$  consisting of linear predictors . Set $M\geq 1$ and $ n\in \llbracket 2,e^M \rrbracket$. Then, there exists an input space $\mathcal{Z}$ with distribution $P$ and a regularization parameter $\lambda_n$,  such that the \textrm{RERM} algorithm $\alg$ (Equation \ref{def:RERM})  satisfies Assumption~\ref{assum:stability-setting} with $L=M$ and  $\beta_n\leq 2 /n$.\\
	Furthermore the \Kfold\ \CV error satisfies,
	\begin{align*}
	\EE\left[\big|\rCV\left[\alg,V_{1:K}\right] - \right. & \left. \risk\left[\alg([n])\right]\big| \right] \geq \\ &2\log\left(\frac{K}{K-1}\right)\left(1-\ffrac{1}{M}\right),
	\end{align*}
	and 
	$$\EE\left[\big|\rCV\left[\alg,V_{1:K}\right] - \risk\left[\alg([n])\right]\big| \right] \leq  2\log\left(\frac{K}{K-1}\right).$$
	
\end{theorem}

\begin{proof}[Sketch of proof]
	First set $\mathcal{Z}=\mathcal{X}\times\mathcal{Y}\subset \rset \times \rset$ and consider the hypothesis class of linear regressors,
	$$ \mathcal{G}=\left\{g_b \mid b\in \rset \right\},$$
	
	where $g_b(x)=bx$ for $x\in \rset$. With the loss $\loss(g_b,o)=(y-b x)^2$, it can be shown that the solution to problem \ref{def:RERM} writes as  $\alg\left(\left[n\right]\right)=g_{b_n}$ with
	\begin{equation}\label{eq:RERM-constrained-form}
		b_n=f(\lambda_n),
	\end{equation}
	 for some decreasing function $f$. Now,  by carefully picking $\mathcal{Z}$ and $\lambda_n$ we construct a problem such that, for any $O\in \mathcal{Z}$ and $T\subset T^{\prime} \subset \left[n\right]$
	
	\begin{enumerate}
		\item $\lambda_n$ is decreasing.
		\item $	\loss\big(\alg\left(T\right),O\big)\geq \loss\big(\alg\left(T^{\prime}\right),O\big)$ .
		\item  $\loss\big(\alg(T^{\backslash i}),O\big) -\loss\big(\alg\left(T\right),O\big) \leq f\left(\lambda_{n}\right)-f\left(\lambda_ {n-1}\right).$
		\item $ f(\lambda_n)-f(\lambda_{n-1})  \leq \ffrac{2}{n}.$
	\end{enumerate} 
The result easily follows  from the three latter facts. For the detailed proof, the reader is deferred to Appendix \ref{sec:proof-LB-RERM}.
\end{proof}

\subsection{Stochastic Gradient Descent}

In this section, we consider the \textrm{SGD} update rule 
\begin{equation}\label{def:SGD}
\forall t \geq 0 \:,\: \alg_{t+1}=\alg_t-\alpha_{t,n} \nabla_\alg \ell\left(\alg_t, X_{i_t}\right),
\end{equation}

where $\nabla_\alg\loss$ denotes the derivative of $\loss$ with respect to the first argument, $i_t$ is the index picked by \textrm{SGD} at step $t$  and $\alpha_{t,n}\geq 0$ is the step size.\\
It is well known that \textrm{SGD} verifies the uniform stability assumption with respect to many losses \citep{hardt16,liu2017,charles18a}. For instance, when the loss function is convex, $\beta$ smooth, and $\sigma$-Lipschitz, it has been shown that SGD algorithm \citep{hardt16} is uniformly stable with parameter $\beta_n \leq  \ffrac{2\sigma^2}{n}\sum_{k=1}^{t}\alpha_{k,n} $. 

%
\begin{remark}
	For a fixed data sequence $\DD$, the output of \textrm{SGD} is random. Hence, the definition of \emph{random uniform stability} is introduced in Appendix \ref{sec:randomized} and is slightly different than Definition \ref{def:unif-stable}. However, most if not all the properties of deterministic uniformly stable  rules  (discussed before)  are preserved by random uniformly stable  learners (see \eg \cite{elisseeff05a,hardt16,liu2017}). 
\end{remark}

\begin{theorem}\label{theo:K-fold-LB-SGD} 
	
	Let $M>1$ be a real number , $2 \leq n\leq e^M$  an integer,  and $t\geq 1$ a maximum number of iterations . Set  the initialization $\alg_0=0$. Then, There exists an input space $\mathcal{Z}=\mathcal{X}\times \mathcal{Y}$, a convex loss function $\loss$ and a sequence of step sizes $(\alpha_{k,n})_{k \leq t}$  such as,
	\begin{align*}
	\EE\left[\big|\rCV\left[\alg_t,V_{1:K}\right] - \right. & \left. \risk\left[\alg_t([n])\right]\big| \right] \geq \ffrac{1}{3}\log\left(\frac{K}{K-1}\right).
	\end{align*}
	
	Furthermore,  $\alg_t$ fulfills Assumption \ref{assum:stability-setting} with  $L=M$ and 
	$$\beta_n\leq \ffrac{3}{n-1}\sum_{k=1}^{t}\alpha_{k,n}\leq  \ffrac{3M}{n-1}.$$
	
\end{theorem}

\begin{proof}
	The proof is deferred to the appendix (section \ref{sec:proof-LB-SGD}).
\end{proof}

\begin{remark}
	In both examples, we suppose that the input is a binary random variable.  Such a restrictive setting serves as a corner case to derive lower bounds \citep{bousquet2020sharper,zhang2022stability}, suggesting the necessity of additional assumptions to ensure the consistency of  \Kfold\ scheme. 
\end{remark}
\begin{remark}[\textsc{boundedness of }$n$]
	With the current assumptions, the lower bounds don't ensure the inconsistency as the sample size $n$ grows to infinity. However, Theorems \ref{theo:K-fold-LB} and \ref{theo:K-fold-LB-SGD} in their current state prove that, one cannot obtain a standard universal vanishing upper bound on the expected estimation error of $K$-fold  ($\mathbb{E}[\mathrm{error}_{\textrm{CV}}]$) which would be valid for all sample spaces ($\mathcal{Z}$), all distributions ($P$) and all stable algorithms. In other words, one cannot construct a function $h(n)$ such that
	\begin{align*}
	\exists& n_0\in \mathbb{N}\:,\:\forall n\geq n_0\:,\: \forall (\mathcal{Z},P),\\ &\forall \mathcal{A}\text{ stable with parameter }\beta_n\leq \frac{1}{n}\:, \mathbb{E}\left[\mathrm{error}_{\textrm{CV}}\right]\leq h(n), 
	\end{align*}
	and  $h(n) \xrightarrow[n \to \infty]{}  0$.\\
	 An additional remark that ought to be made: if we relax Assumption \ref{assum:stability-setting} into : "the loss function  $\ell(\mathcal{A}\left([n]\right),\cdot)$ is $L\log(n)$-bounded and the algorithm $\mathcal{A}$ is $\frac{\log(n)}{n}$ uniformly stable", then all the upper bounds from the present paper remain valid (up to a $\log(n)$ factor), in particular the upper bound on the bias-corrected $K$-fold estimation error (Corollary \ref{cor:consistentCorrectedKfold} below). Furthermore, under this new assumption, the lower bound in Theorem \ref{theo:K-fold-LB-SGD} is valid $\forall n \in \mathbb{N}$, which yields the inconsistency of $K$-fold \textrm{CV}. More precisely, Theorem \ref{theo:K-fold-LB-SGD} becomes
	\begin{align*}
	\exists & (\mathcal{Z},P), \forall  n_0\in \mathbb{N}\:,	\:\forall n\geq n_0\:,\:\\& \exists \mathcal{A}\text{ stable with parameter }\beta_n\leq \frac{\log(n)}{n},\, \mathbb{E}\left[\mathrm{error}_{\textrm{CV}}\right]\geq \kappa,
	\end{align*}
	with $\kappa=\frac{1}{3}\log\left(\frac{K}{K-1}\right)$.
	One must note that the additional $\log(n)$ factor in the upper bound weakens the tightness of the lower bound, however this is compatible with  existing lower bounds from the stability literature \citep{bousquet2020sharper}.  
\end{remark}

\begin{discussion*}
	Theorems~\ref{theo:K-fold-LB} and \ref{theo:K-fold-LB-SGD} reveals the suboptimality of \Kfold\:CV in the general stability framework for risk estimation. For the purpose of model selection, which is arguably a harder problem,
	existing works have shown the suboptimality of \Kfold\ \CV. 
	For example, in a regression framework, under the linear model assumption, \cite{Yuhong2007} (See also \cite{yang2006} for classification problems) has shown that the usual K-fold procedure may not select the best model. For  efficient model selection, where the performance is measured in terms of generalization risk of the selected model, \cite{arlot2008} shows (See Theorem~1 in the latter reference) that \Kfold\ \CV can be suboptimal, \ie an example is provided  where the risk ratio between the selected model and the optimal one is uniformly greater (for all $n$) than $1+\kappa$  for some $\kappa>0$. In Theorem~\ref{theo:K-fold-LB} (\emph{resp.} \ref{theo:K-fold-LB-SGD}) of the present paper we show a stronger result, in so far as we consider the easier task of risk estimation, and we show that $\Kfold\ \CV$ with fixed $K$ does not enjoy sanity-check guarantees because for any $n$, there exists a regression (\emph{resp.} optimization) problem  where the error is at least $2\log(\frac{K}{K-1})$ \big(\emph{resp.} $\frac{1}{3}\log\frac{K}{K-1}$\big). It is worth mentioning at this stage that even if the uniform stability ensures the low variance of \Kfold\ (\cite{kumar2013near,bayle2020cross}, etc.), it is  not sufficient to control the bias. 
\end{discussion*}

In the next section, we show that adding correction terms
\citep{BURMAN89} to \Kfold\ \CV addresses the inconsistency issues
underlined by Theorems~\ref{theo:K-fold-LB} and \ref{theo:K-fold-LB-SGD}. The resulting \CV\ scheme enjoys both the
computational efficiency of the \Kfold\ (compared with the \loo) and finite sample guarantees  comparable to those of the \loo\ in view of the upper bound
stated in Theorem~\ref{theo:CVC-UB-det}.

\section{BIAS CORRECTED \textrm{K-FOLD}\:WITH STABLE LEARNERS}
\label{sec:BKF-UB}

A key ingredient of the lack of guarantee regarding the \Kfold\ risk estimate is its bias for finite sample sizes, an issue pointed out by \cite{BURMAN89} who proposes a corrected version of the standard \Kfold\ with a reduced bias, see also~\cite{burman1990estimation} for applications to model selection. In the present work we follow in the footsteps of \cite{BURMAN89} and consider the same  corrected \CV\ estimate of the generalization risk 
\begin{equation}
\begin{aligned}
\label{def:risk-CV-corr}
\rCVC\left[\alg,V_{1:K}\right] & =  \rCV\left[\alg,V_{1:K}\right]\\
& \hspace{-15mm}  + \ER\left[\alg([n]),[n]\right]
-\frac{1}{ K} \sum_{j= 1}^{K} \ER\left[\alg(T_j),[n]\right]. 
\end{aligned}
\end{equation}

The correcting term (second line in the above display) is the average  difference  between   $\risk\left[\alg([n])\right]$ and the empirical risks $\risk\left[\alg(T_j)\right]$'s. 
In \cite{BURMAN89} the analysis is carried out in an asymptotic
framework and focuses on the asymptotic bias and variance of the
estimator for different \CV\ schemes.  The results are obtained under
differentiability assumptions regarding the loss function (See also
the appendix section in \cite{fushiki2011estimation} where these
assumptions are explicitly listed) which are typically not satisfied
by SVM, stochastic gradient methods or $L^1$-regularized risk minimization algorithms. In
\cite{burman1990estimation} the working assumption is that of a linear
model and the specific task is ordinary linear
regression.  
In contrast, we conduct here a non asymptotic analysis (valid for any
sample size) which applies to any uniformly stable
algorithm. 
\begin{theorem}
\label{theo:CVC-UB-det}
Suppose that Assumption~\ref{assum:stability-setting} holds. Then, we
have, with probability $1-6\delta$,
 \begin{align*}
 	\bigg|\rCVC\left[\alg,V_{1:K}\right]-\TR\big[&\alg([n])\big]\bigg|\leq  2\left(\beta_n+\beta_{n_T}\right)\\ &+3(4\beta_{n_T}n_T+2L)\sqrt{\ffrac{\log(1/\delta)}{2n}}.
 \end{align*}

\end{theorem}

\begin{proof}[Sketch of proof]
  By simple algebra write the corrected \CV\ estimator as
  \begin{align*}
	\rCVC\big[\alg,&V_{1:K}\big] = \ER\left[\alg([n]),[n]\right]\nonumber \\+ &\ffrac{n_T}{nK}\sum_{j=1}^{K}\left[\ER\left[\alg(T_j),V_j\right]-\ER\left[\alg(T_j),T_j)\right]\right] .
\end{align*}
	From this,  deduce the following error decomposition
	\begin{equation}
          \begin{aligned}
            \rCVC\left[\alg,V_{1:K}\right]- &\risk\left[\alg[n]\right] = 
            \DevAll \\
            &+\ffrac{n_T}{n}(\DevCV-\DevT), \label{ineq:error-decom-cvc}
          \end{aligned}
	\end{equation}
	where $\DevCV$ is  defined in \eqref{def:dev_CV} and 
	\begin{align}
	\hspace{-1mm}\DevAll &= \ER\left[\alg[n],[n]\right]-\risk\left[\alg[n]\right]\label{def:dev_all}, \\
	\hspace{-1mm}\DevT &=\ffrac{1}{K}\sum_{j=1}^{K}\left( \ER\left[\alg(T_j),T_j\right]-\risk\big[\alg(T_j)\big]\right).\label{def:dev_train}
	\end{align}
        Notice the absence of a bias term in the above display, contrarily to the error decomposition (\ref{ineq:error-decomp}) for the standard \CV\ estimate in the proof of Theorem~\ref{theo:CV-stable-bound}. 
	The remaining technical arguments for bounding each term of the decomposition above are gathered in the supplement.  
        Namely, using McDiarmid's inequality (Proposition \ref{prop:Mcdiarmids-ineq}), we derive concentration bound both for $\DevCV$ as stated in Lemma \ref{lemma:Dev-CV-UB} and for $\DevT$ as given in Lemma \ref{lemma:Dev-train-UB}. Finally the deviations of the empirical risk $\DevAll$ are controlled using Remark \ref{remark:bousquet-UB-gen}, a consequence of Lemma  \ref{lemma:Dev-train-UB},  in the supplement.
\end{proof}

Theorem~\ref{theo:CVC-UB-det} yields immediately a consistent upper
bound for the \Kfold\:\CV\ error when $\beta_n\leq C/n$ for some
$C>0$. For the proof it suffices to notice that $n_T=\frac{K-1}{K}n$
for the \Kfold\:scheme and to use that $ { (2K-1) } / {(K-1) } \leq 3$
whenever $K\geq 2$.
\begin{corollary}\label{cor:consistentCorrectedKfold}
	Suppose that Assumption~\ref{assum:stability-setting} holds with $\beta_n \leq \ffrac{C}{n}$, for some $C>0$. Then, for $K\ge 2$, the error of the  corrected \Kfold\:\CV\  estimate  of the generalization risk satisfies with probability at least $1-6\delta$, 
	\begin{align*}
		\bigg|\ER^{corr}_{CV}\left[\alg,V_{1:K}\right]-\TR\big[&\alg([n])\big]\bigg|\leq  \ffrac{6C }{n} \\ &+3(4C+2L)\sqrt{\ffrac{\log(1/\delta)}{2n}}.
	\end{align*}
\end{corollary}
\begin{discussion*}
  Corollary~\ref{cor:consistentCorrectedKfold}  confirms the relevance of the bias corrected \Kfold\:for risk estimation 
 in the broad context of uniformly  stable learners satisfing $\beta_n\leq \frac{C}{n}$, such as SVM, stochastic gradient methods  or regularized empirical risk minimizers.
  
	
\end{discussion*}

Our main results Theorem~\ref{theo:CVC-UB-det} and
Corollary~\ref{cor:consistentCorrectedKfold} concern the problem of
risk estimation by means of \Kfold~\CV. However in practice, \CV\ is
widely used in the context of model- or parameter \emph{selection}. It is
precisely the purpose of the next section to illustrate how our
results shed light on such practice. Namely we consider the practical
problem of selecting a penalty parameter among a finite collection of
candidates for Support Vector Regression and Classification.





%


\section{APPLICATION TO HYPER-PARAMETER SELECTION AND NUMERICAL EXPERIMENTS} 
\label{sec:model-selec}

Selecting a penalty parameter in regularized empirical risk
minimization algorithms such as SVM's may be viewed as a particular
instance of a model selection problem, where one identifies a model
with an algorithm equipped with a particular choice of regularization
parameter. 
We start-off this section with a brief introduction to this topic.  We provide minimal theoretical guarantees (Proposition~\ref{theo:model-selec-CVC}) regarding model selection within a  finite collection of models by means of the debiased \Kfold\xspace procedure. We describe our experimental setting in Subsections~\ref{sec:experiments-setting},~\ref{sec:datasets} and we report our results in Section~\ref{sec:expe-results}.

\subsection{Cross-Validation for Parameter Selection}\label{sec:model-select-context}
Following the terminology of \cite{arlot2016VFchoice}, we
consider here the problem of \emph{efficient} model selection, aiming at selecting a model for which the generalization risk of the
learnt predictor is close to the smallest possible risk. The goal of efficient model selection is in general easier to attain than \emph{identification}  of the best possible model. 
  It is a known fact that  cross-validation is in general sub-optimal for
  model  \emph{identification} purpose. 
  A major reason for this is that different models (or algorithms) may have comparable performance. However if one aims only at selecting a model for which the generalization risk is close to that of the optimal one, lack of identifiability is not necessarily an issue anymore,  which is precisely the approach we take here.    


Model selection is a prominent topic in statistical learning theory which is by far too broad to be extensively covered here. We refer the reader to the monograph of ~\cite{massart2007concentration}. Cross-validation is one of several  possible candidate methods for  this problem, which has been the subject of a wealth of literature as discussed in the introduction, see also~\cite{arlot2010survey} for a review.     

Given a family of models (or algorithms) $\alg^{(m)}$ indexed by $m \in \mathcal{M}$ and a dataset $\DD$ of size $n$, an optimal model $\alg^{(m^*)}$ called an oracle is any model  such that
\begin{equation}\label{def:oracle}
	 \alg^{(m^*)} \in \argmin_{m\in \mathcal{M}}\risk\left[\alg^{(m)}([n])\right].
\end{equation}
Since the true risk is unknown, an empirical criterion must be used instead to select an efficient model. One of the most popular tools for model selection is \Kfold\ \CV. However, as discussed earlier, the latter procedure may not be consistent. To tackle this issue, we propose to use the corrected \Kfold\: and select a model $\alg^{\widehat m}$ such that
\begin{equation}\label{def:opt-Kfold}
	 \alg^{(\widehat m)} \in \argmin_{m\in \mathcal{M}}\ER^{corr}_{\mathrm{Kfold}}\left[ \alg^{(m)},V_{1:K}\right]. 
\end{equation}
Proposition~\ref{theo:model-selec-CVC} provides a sanity check guarantee regarding the consistency of the corrected \Kfold\xspace\CV procedure for this purpose, in the form of an  upper bound in probability for  the excess risk $\risk\left[\alg^{(\hat m)}[n]\right] - \risk\left[\alg^{(m^*)}([n]) \right]$. The result applies to  the case where the family of models $\mathcal{M}$ is finite.
  
\begin{proposition}\label{theo:model-selec-CVC}
	Let $(\alg^{(m)})_{m\in \mathcal{M}}$ be a family of algorithms where each learner $\alg^{(m)}$ is $(\beta_{m,t})_{1\leq t \leq n}$ uniformly stable with respect to a loss function $0\leq \loss(g,O)\leq L$. Additionally, assume that, 
	$$ \forall m \in \mathcal{M} \:;\: \beta_{m,t} \leq \ffrac{M}{t},$$
	for some $M>0$. Then one has, with probability at least $1-6\delta$,
	\begin{align*}
		\risk\big[\alg^{(\hat m)}([n])\big]-\TR\big[&\alg^{(m^*)}([n])\big]\leq  \ffrac{12M }{n}\\ &+6(4M+2L)\sqrt{\ffrac{\log(\abs{\mathcal{M}}/\delta)}{n}}.
	\end{align*}
\end{proposition}

\begin{proof}
  The proof consists in applying a union bound combined with the
  exponential tail bound of Theorem~\ref{theo:CVC-UB-det}, which
  yields a multiplicative constant which only depends logarithmically
  on the number of models. The details are gathered in the
  supplementary material,
  Section~\ref{prop:uniform-CVC-UB}. 
\end{proof}
\begin{discussion*}[bounded stability parameters $\beta_{m,t}$'s]
	The assumption $\beta_{m,t} \leq \ffrac{M}{t}$ for all $m$  is indeed verified in many applications. For example, in regularized SVM, where each learner $\alg^{(m)}$ is trained using a regularization parameter $\lambda_m$, \cite{bousquet2002stability} show that $\beta_{m,t}=\ffrac{1}{t}\sqrt{\frac{C}{\lambda_m}}$ where  $C$ is a positive constant. Thus, if one performs a grid search for $\lambda_m$ on a grid $[a,b]$ with $a>0$, then $\beta_{m,t}\leq \ffrac{1}{t}\sqrt{\frac{C}{a}}$ for any $\lambda_m \in [a,b]$. In other words, since the search space is generally bounded, the boundedness assumption regarding the $\beta_{m,t}$'s  is not too restrictive in practice.
\end{discussion*}

\subsection{Support Vector Machines and Experimental Setting}\label{sec:experiments-setting}

The aim of our experiments is to illustrate empirically the added
value of the corrected \Kfold\ compared with the standard one in
terms of efficiency in model selection. In other words, we perform
model selection with the \Kfold\ and the corrected version that we
promote and we compare the generalization risks  of the selected
trained models.

We 
consider a finite family of  SVM regressors and classifiers trained with a regularization parameter $\lambda_m$ ranging in a finite grid in an interval $[a,b]$, where  $m\in\mathcal{M}$ a finite index set.
Namely we set $[a,b]=[0.1,100]$, and the grid is constructed with a  constant step size equal to $\Delta = 0.1$, so that $$(\lambda_m)_{m\in \mathcal{M} }=\{a+j\Delta \mid 0 \leq j < 1000\}.$$ 
 In this SVM framework, 
$$\alg^{(m)}(T)=\argmin_{f\in \mathcal{F}}\ffrac{1}{n_T} \sum_{i\in T}\loss(f,O_i)+\lambda_m \lVert f \rVert^2_k$$
where $\mathcal{F}$ is  a reproducing kernel Hilbert space with kernel $k$. The kernel $k$ is chosen here as the sigmoid
kernel $\tanh( \tau \langle x, \rangle )$. Following standard
practice we set $\tau=\frac{1}{d}$, where $d$ is the dimension of the
dataset.
We  use the quadratic loss for regression problems and  the hinge loss $\loss(g,(x,y))= (1-yg(x))_{+}$ for classification, where $(f(x))_+=max(0,f(x))$.  Since the training datasets are bounded, we may consider that  both these losses as bounded as well. For both these losses the algorithm $\alg^{(m)}$ is $\ffrac{C}{n}$ uniformly stable (see \cite{bousquet2002stability} for further details).
The assumptions of Theorem~\ref{theo:model-selec-CVC} are thus satisfied, as pointed out in the discussion following the theorem's statement.

\subsection{Datasets}\label{sec:datasets} 
Eight reference datasets from UCI \footnote{\url{https://archive.ics.uci.edu}} are considered: four classification datasets and four regression datasets listed below.
\paragraph{Regression datasets } \emph{Real estate valuation} (REV, 414 house price of unit area with 5 covariates); \emph{QSAR fish toxicity } (906 toxic chemical concentration with 6 attributes); \emph{Energy efficiency} (EE, 768 heating loads with 8 features.); \emph{Concrete Compressive Strength} (CS, 1030 measure of the compressive strength with 8 attributes).
\paragraph{Classification datasets }\emph{Ionosphere dataset} (IO, 351 radar returns with 34 attributes), \emph{Raisin dataset} (RS, 900 Keciman/Besni raisin with 7 attributes), \emph{Audit risk dataset} (AR, 777 firm evaluation (fraudulent/non fraudulent) with 18 risk factors) and \emph{QSAR bio degradation Data Set } (BIODEG, 1055 chemicals categorization with 12 descriptive features).

For  each dataset one third of the data  are removed ($\mathcal{S}$)  and reserved for testing, \ie for evaluating the generalization risk of the  model selected using the remaining two thirds ($\DD$). 
Let  $\hat m_{\mathrm{Kf}}$ (\emph{resp.} $\hat m_{\mathrm{Kf-corr}}$)  denote the model selected using \Kfold \ \CV\ (\emph{resp.} corrected \Kfold\ \CV)   on the train set $\DD$. In other words
\begin{equation*}
\begin{aligned}
		 \alg^{(\widehat m_{\mathrm{Kf-corr}})}&=\argmin_{m\in \mathcal{M}}\ER^{corr}_{\mathrm{Kfold}}\left[ \alg^{(m)},V_{1:K}\right], \\
		 \alg^{(\widehat{m}_{\mathrm{Kf}})}&=\argmin_{m\in \mathcal{M}}\ER_{\mathrm{Kfold}}\left[ \alg^{(m)},V_{1:K}\right].
\end{aligned}
\end{equation*}
In the end, in line with Section~\ref{sec:model-selec}, the performance of both models are compared in terms of the mean squared error (or hinge loss for classification)  and its  estimated  standard deviation on the test set $\mathcal{S}$.

\subsection{Numerical Results}\label{sec:expe-results}

We use the implementation provided by the python library
\texttt{scikit-learn}.  
Tables~\ref{table:regression-results} and~\ref{table:classification-results} gather the results obtained respectively with the regression and classification datasets, for different numbers of folds $K$ varying between $3$ and $5$. In all  cases, the model selected by the bias corrected \Kfold\: has a lower generalization risk than the one selected by the standard \Kfold.  As expected,  the standard  \Kfold\:procedure behaves generally  better with  $K=5$ than with $K=3$. Indeed larger values of $K$ decrease the bias of the \Kfold\ \CV. The benefit of the bias correction is thus all the more important for small values of $K$. 

\begin{table}[!h]
	\caption{Regression mean squared errors for the \Kfold\: and the bias corrected \Kfold\:
		on various data sets. Estimated standard deviations are reported between parentheses.}
	\label{table:regression-results}
	\vskip 0.15in
	\begin{center}
		\begin{small}
			\begin{sc}
				\begin{tabular}{@{}rccr@{}}
					\toprule
					Dataset & K-fold & Bias corrected K-fold \\ 
					\midrule
      				 REV; K=3 & 74.198 (12.57) & \textbf{68.958 (11.63)} \\ 
					 K=4 & 74.189 (12.48)  & \textbf{68.958 (11.63)} \\ 
					 K=5 & 73.359 (12.38) & \textbf{68.958 (11.63)} \\ 
					\midrule
					EE; K=3 & 15.501  (1.81) & \textbf{14.405  (1.86)} \\ 
					K=4 & 15.825 (1.84) & \textbf{14.350 (1.77)} \\ 
					K=5 & 14.730 (1.73) & \textbf{14.298 (1.79)} \\ 
					\midrule
					QSAR; K=3 & 1.183 (0.16) & \textbf{1.035 (0.14)} \\ 
				    K=4 & 1.112  (0.15) & \textbf{1.035 (0.14)} \\ 
				    K=5 & 1.112 (0.15) & \textbf{1.035 (0.14)} \\ 
					\midrule
					CS; K=3 & 146.881 (13.81) & \textbf{126.492(10.23)} \\ 
				        K=4 & 144.195  (13.16) & \textbf{124.205 (9.46)} \\ 
						K=5 & 137.060  (11.48) & \textbf{123.641 (9.30)}  \\ 
					\bottomrule
				\end{tabular}
			\end{sc}
		\end{small}
	\end{center}
	\vskip -0.1in
    \end{table}

\begin{table}[!h]
	\caption{Hinge losses for the \Kfold\: and the bias corrected \Kfold\:
		on various data sets. Estimated standard deviations are reported between parentheses.}
	\label{table:classification-results}
	\vskip 0.15in
	\begin{center}
		\begin{small}
			\begin{sc}
				\begin{tabular}{@{}rccr@{}}
					\toprule
					Dataset & K-fold & Bias corrected K-fold \\ 
					\midrule
					RS; K=3 & 0.470 (0.048) & \textbf{0.419 (0.038)}\\ 
					K=4 & 0.420 (0.039) & \textbf{0.418 (0.039) } \\ 
					K=5 & 0.420  (0.039) & \textbf{0.419 (0.038)} \\ 
					\midrule
					IO; K=3 & 0.454 (0.081) & \textbf{0.414 (0.072)} \\ 
					K=4 & 0.447 (0.092) & \textbf{0.425 (0.072)} \\ 
						K=5 & 0.477 (0.091)  & \textbf{0.464 (0.095)} \\ 
					\midrule
					BIODEG; K=3 & 0.361 (0.039)& \textbf{0.357 (0.036) } \\ 
						    K=4 & 0.363 (0.037) & \textbf{0.357 (0.036)}   \\ 
							K=5 & 0.381 (0.041) & \textbf{0.357 (0.036)} \\ 
				    \midrule	
					AR; K=3 & 0.112 (0.023) & \textbf{0.109 (0.021)} \\ 
						K=4 & 0.108 (0.027) & \textbf{0.105 (0.025)} \\ 
						K=5 & 0.107 (0.025) &  \textbf{0.106 (0.024)} \\ 
					\bottomrule
				\end{tabular}
			\end{sc}
		\end{small}
	\end{center}
	\vskip -0.1in
\end{table}



\section{Conclusion}
This paper demonstrates the limitations of the standard \Kfold\
procedure for risk estimation \emph{via} a lower bound on its error
with a uniformly stable learner. We show that the corrected
version of the \Kfold \ for uniformly stable algorithms does not
suffer the same drawbacks through a sanity-check upper bound
and leverage this result to obtain guarantees regarding efficient model
selection. This paves the way towards two possible research
directions. A relevant follow-up would be to relax our uniform stability
assumption in order to cover  still a wider class of algorithms such as k-nearest-neighbors \citep{DEvroy-79}, Adaboost \citep{FREUND97} and Lasso regression \citep{celisse2016stability}. A second promising avenue would be to consider an extension of the \Kfold\
\emph{penalization} proposed by~\cite{arlot2016VFchoice} to the class of
stable learners.


\bibliography{biblio_crossval}

\appendix 

\onecolumn

\section{MAIN TOOLS}\label{sec:main-tools}
First we recall McDiarmid's inequality (Theorem 3.1 in \citet{McDiarmid98conc}).
\begin{prop}[McDiarmid's inequality]\label{prop:Mcdiarmids-ineq}
	let $Z=f(\DD)$ for some measurable function $f$ and define 
	\[ \Delta_l(\DD,O')= f(\DD)-f(\DD^l), \]
	where  $\DD^{l}$ is obtained by replacing the $l$'th element of $\DD$ by a sample $O' \in \mathcal{Z}$.\,In addition,\,suppose that 
	\[  \forall l \in [n] \:,\:\sup_{\DD\in \mathcal{Z}^n}\sup_{O' \in \mathcal{Z}} \lvert  \Delta_l(\DD,O') \rvert  \leq c_l.\]
	Then for any $t\geq 0$,
	\[ \PP(Z-\EE(Z)\geq t ) \leq \exp{\left( \frac{-2t^2}{\sum_{l=1}^n c_l^2}\right)}.\]
\end{prop}

The following lemma guarantees that Fact \ref{fact:mask-property} is verified for \Kfold\ \CV.
\begin{lemma}
	\label{lemma:CV-scheme-property}
	For the $K$-fold procedure,\,the training samples 
	$T_{1:K}$ and validation samples $V_{1:K}$ satisfy Fact \ref{fact:mask-property} i.e 
	\begin{equation*}
	\frac{1}{K}\sum_{j=1}^{K}\ffrac{\indic{l \in T_j} }{n_{T}}= \frac{1}{K}\sum_{j=1}^{K}\ffrac{\indic{l \in V_j} }{n_{V}}=\frac{1}{n} \quad \forall l \in \llbracket1,n\rrbracket.
	\end{equation*}	
\end{lemma}
\begin{proof}

	The validation sets  of \Kfold\ verify the following property
	\begin{equation}\label{eq:K-fold-partition}
	\bigcup\limits_{j=1}^K V_j=\llbracket1,n\rrbracket \: and \: V_j\bigcap V_k=\varnothing\:, \:\forall j\neq k \in \llbracket1,K\rrbracket.
	\end{equation}
	
	Under the hypothesis that $card(V_j)=n_{V}$  for all the validation sets, \eqref{eq:K-fold-partition} implies\,that
	\begin{equation}
	\label{}
	n=\sum_{j=1}^{K}card(V_j)=Kn_{V}. 
	\end{equation}

	Furthermore,\,under \eqref{eq:K-fold-partition},\,an index $l\in \llbracket 1,n \rrbracket$ belongs to a unique validation test $V_j'$ and to all the train sets $T_j=V_j^c$ with $j\neq j'$.\,Hence,\,we have
	
	\begin{equation*}
	\begin{cases}
	\sum_{j=1}^{K}\indic{l \in T_j} = K-1, \\
	\sum_{j=1}^{K}\indic{l \in V_j} = 1.
	\end{cases}  
	\end{equation*}
	Using \eqref{eq:K-fold-partition} and the fact that $n_{T}=n-n_{V}=(K-1)n_{V}$ yields the desired result.

\end{proof}

\section{INTERMEDIATE RESULTS}\label{sec:intermediate}

First we provide a concentration bound for $\DevCV$ which has been defined in \eqref{def:dev_CV} as
$$\DevCV= \rCV\left[\alg,V_{1:K}\right]-  \risk_{CV}\left[\alg,V_{1:K}\right].$$

\begin{lemma}\label{lemma:Dev-CV-UB}
 Suppose that Assumption~\ref{assum:stability-setting} holds. Then, we have
	\begin{equation*}
	\PP(\abs{\DevCV} \geq t) \leq \exp{\left(\ffrac{-2nt^2}{(4\beta_{n_{T}}n_{T}+2L)^2} \right)}.
	\end{equation*}

\end{lemma}
\begin{proof} 
	Let  $O'\in \mathcal{X}$ be  an independent copy of $O_1,O_2,\dots,O_n$, for any $l\in \llbracket1,n\rrbracket$ define 
	\[\Delta_l(\DD,O')=\big\lvert \DevCV(\DD)-\DevCV(\DD^l)\big\rvert, \]
	where $\DD^{l}$ is obtained by replacing the $l$'th element of $\DD$ by $O'$. Now we derive an upper bound on $\PP(\abs{\DevCV} \geq t)$ using Proposition \ref{prop:Mcdiarmids-ineq}. Namely, we will bound the maximum deviation of $\Delta_l$ by $\Delta_l \leq  \ffrac{4\beta_{n_{T}} n_{T}}{n}+\ffrac{L}{n}$. To do so, write 
	\begin{align*}
	\DevCV&=\rCV\left[\alg,V_{1:K}\right]-  \risk_{CV}\left[\alg,V_{1:K}\right]\\
	&=\frac{1}{Kn_{val}}\sum_{j=1}^K\sum_{i \in V_j} \bigg(\loss(\alg (T_j),O_i)-\EE_{O}\big[\loss(\alg(T_j),O) \mid \DD_{T_j})\big]\bigg)\\
	&=\frac{1}{Kn_{val}}\sum_{j=1}^K\sum_{i \in V_j}h(\alg(T_j),O_i),
	\end{align*}
	where the last is used to define $h$. For a training set $T_j\subset [n]$, let $\alg(T_{j,l}) $ denote the algorithm $\alg$ trained on the sequence $\DD_{T_j,l}=\{O_i \in \DD^{l} \mid i \in T_j\}$. Note that, for all $ \DD_{T_j},o \in \mathcal{Z}^{n_{T}}\times \mathcal{Z}$ one has 
	\begin{equation}
	\label{"train_set_dev"}
	\begin{cases}
	\loss\big(\alg(T_j),o\big)=\loss\big(\alg(T_{j,l}),o\big)\: \text{if} \: l \notin T_j\,,\\
	\bigg \lvert \loss\big(\alg(T_j),o\big)-\loss\big(\alg(T_{j,l}),o\big) \bigg\rvert \leq 2\beta_{n_{T}} \: \text{otherwise}.
	\end{cases}
	\end{equation}
	The first equation follows from the fact that $\DD_{T_j,l}=\DD_{T_j}$ if $l\notin T_j$, indeed, if the training set $\DD_{T_j}$ doesn't contain the index $l$ then changing the $l$'th element of $\DD$ won't affect $\DD_{T_j}$.\,The second inequality is obtained using the \emph{uniform} stability of $\alg$. Furthermore, using Equation \ref{"train_set_dev"} write
	\begin{equation}
	\label{"expectation_train_set_dev"}
	\begin{cases}
	\EE\bigg[\loss\big(\alg(T_j),O\big)\mid \DD_{T_{j}}\bigg]=\EE\bigg[\loss\big(\alg(T_{j,l}),O\big)\mid \DD_{T_{j,l}}\bigg]\: \text{if} \: l \notin T_j\\
	\bigg \lvert \EE\big[\loss\big(\alg(T_j),O\big)\mid \DD_{T_j}\big]- \EE\big[\loss\big(\alg(T_{j,l}),O\big)\mid \DD_{T_{j,l}}\big] \bigg\rvert \leq 2\beta_{n_{T}} , \: \text{otherwise}.
	\end{cases}
	\end{equation}
	
	Combining \eqref{"expectation_train_set_dev"} and \eqref{"train_set_dev"} gives
	\begin{equation}
	\label{eq:key-DCV-Mcdiarmid}
	\begin{cases}
	\abs{\un{l \in T_j}(h(\alg(T_j),O_i))-h(\alg(T_{j,l}),O_i))} \leq 4\beta_{n_{T}} \,,\\
	\abs{\un{l \notin T_j}(h(\alg(T_j),O_i))-h(\alg(T_{j,l}),O_i)) }\leq 2L\un{i=l}, 	
	\end{cases}
	\end{equation}
	so that
	\begin{align*}
	\abs{\Delta_l(\DD,O')}&\leq  \frac{1}{Kn_{val}}\sum_{j=1}^K\sum_{i \in V_j}\abs{h(\alg(T_j),O_i))-h(\alg(T_{j,l}),O_i))} \\
	(\text{From the fact that }[n] \backslash T_j =V_j)\ &=\frac{1}{Kn_{val}}\sum_{j=1}^K\sum_{i \in V_j}\abs{h(\alg(T_j),O_i))-h(\alg(T_{j,l}),O_i))}\left(\un{l\in T_j}+\un{l\in V_j}\right)\\
	(\text{By Equation \ref{eq:key-DCV-Mcdiarmid}} )\  &\leq \ffrac{4\beta_{n_{T}}}{K}\sum_{j=1}^{K}\un{l \in T_j} +\ffrac{2L}{n_{val}K}\sum_{j=1}^{K}\un{l \in V_j}\\
	(\text{By Fact \ref{fact:mask-property}} )\   &\leq \ffrac{4\beta_{n_{T}} n_{T}}{n}+\ffrac{2L}{n}.
	\end{align*}
	Using Mcdiarmid's inequality ( Proposition \ref{prop:Mcdiarmids-ineq}) gives
	\begin{equation*}
	\PP(\DevCV \geq t) \leq \exp{\left(\ffrac{-2nt^2}{(4\beta_{n_{T}}n_{T}+2L)^2} \right)}.
	\end{equation*}
	Symmetrically, one has,
	\begin{equation*}
		\PP(\DevCV \leq -t) \leq \exp{\left(\ffrac{-2nt^2}{(4\beta_{n_{T}}n_{T}+2L)^2} \right)}.
	\end{equation*}
	Thus,
	$$ \PP(\abs{\DevCV} \geq t) \leq 2\exp{\left(\ffrac{-2nt^2}{(4\beta_{n_{T}}n_{T}+2L)^2} \right)}, $$
	which is the desired result. 
\end{proof}

In the next lemma we obtain  a similar concentration bound for the term $\DevT$ defined in Eq~\ref{def:dev_train} as
$$	\DevT =\ffrac{1}{K}\sum_{j=1}^{K} \left(\ER\left[\alg(T_j),T_j\right]-\risk\big[\alg(T_j)\big]\right).$$

\begin{lemma}\label{lemma:Dev-train-UB}
	Suppose that Assumption~\ref{assum:stability-setting} holds. Then, one has,
	\begin{equation*}
		\PP(\abs{\DevT}   \geq t + 2\beta_{n_{T}}  ) \leq \exp{\left(\ffrac{-2nt^2}{(4\beta_{n_{T}}n_{T}+2L)^2} \right)},
	\end{equation*}
	where  $L$ is the upper bound of the cost function defined in Assumption~\ref{assum:stability-setting}.
\end{lemma}

	\begin{proof}
Though the proof bears resemblance with the one of Lemma \ref{lemma:Dev-CV-UB}, we provide the full details details for compeleteness.  We use McDiarmid's inequality with $f(\mathcal D) = \DevT  $, that is,
\begin{align*}
f(\mathcal D) &= \ffrac{1}{K}\sum_{j=1}^{K} \left( \ER\left[\alg(T_j),T_j\right] - \risk\big[\alg(T_j)\big]\right) \\
f(\mathcal D^l)  &= \ffrac{1}{K}\sum_{j=1}^{K}  \left( \ER\left[\alg(T_{j,l}),T_{j,l}\right] - \risk\big[\alg(T_{j,l} )\big]\right).
\end{align*}
Since for $l\notin T_j $, $T_j = T_{j,l}$, we find that
\begin{align*}
	\abs{\Delta_l(\DD,O')}	& \leq   \frac{1}{Kn_{T}}  \ \sum_{j=1}^K \indic{l \in T_j}  \left( \ER\left[\alg(T_j),T_j\right] -  \ER\left[\alg(T_{j,l} ),T_{j,l} \right]   + 
	\risk\big[\alg(T_{j,l} )\big] - \risk\big[\alg(T_j)\big] \right)\\
	& \leq   \frac{1}{Kn_{T}}  \ \sum_{j=1}^K \indic{l \in T_j}  \sum_{i \in T_j}\abs{   h (\alg (T_j),O_i) - h(\alg (T_{j,l} ), O_i^l)  } \,,
	\end{align*}
	with $ h(\alg(T), o  ) = \loss(\alg (T), o) - \EE_{O}\big[\loss(\alg(T),O) \mid \DD_{T } \big] $ and $   (O_i^l)_{i=1,\ldots, n }$ is the same as $(O_i)_{i=1,\ldots, n }$ except the $l$-th element, $O_l$, which is replaced by $O'$. Whenever $ l \in T_j$, it holds that 
	\begin{align*}
	 | \loss (\alg (T_j),O_i) -  \loss (\alg (T_{j,l} ), O_i^l) |  &=  |\loss (\alg (T_j),O_i) - \loss (\alg (T_{j,l} ), O_i) | \indic {i\neq l} +|\loss (\alg (T_j),O_l) - \loss (\alg (T_{j,l} ), O') | \indic {i =  l} \\
	&\leq   2\beta _{n_T}  + L\indic {i =  l}  \,,
	\end{align*}
and that 
\begin{align*} 
 \EE\big[ | \loss(\alg(T_j),O) -\loss(\alg(T_{j,l}),O)  | \mid \DD_{T }, O' \big] 	&\leq  2\beta _{n_T} .
	\end{align*}	
It follows from the definition of $h$ that
\begin{align*}
	 | h (\alg (T_j),O_i) - h(\alg (T_{j,l} ), O_i^l) |  	&\leq   4\beta _{n_T}  + 2L\indic {i =  l} .
	\end{align*}	
		 By using 
		 the identity $\ffrac{1}{K}\sum_{j=1}^{K}  {\indic{l \in T_j} }=\ffrac{n_{T}}{n}$ we get
		 \[\Delta_l(\DD,O')  \leq \ffrac{4\beta_{n_{T}} n_{T}}{n}+\ffrac{2L}{n}. \]
Thus by Mcdiarmid's (Proposition \ref{prop:Mcdiarmids-ineq}), we obtain, for $Z = f(\mathcal D) - \mathbb E [  f(\mathcal D) ] $, that
		\begin{equation*}
			\PP( Z \geq t) \leq \exp{\left(\ffrac{-2nt^2}{(4\beta_{n_{T}}n_{T}+2L)^2} \right)}.
		\end{equation*}
		Symmetrically,  the event $ - Z \geq t$ is subject to the same probability bound. It follows that 
		\begin{equation}\label{ub_proba_useful}
			\PP( | Z |  \geq  t) \leq 2\exp{\left(\ffrac{-2nt^2}{(4\beta_{n_{T}}n_{T}+2L)^2} \right)}\,.
		\end{equation}
		 To derive an upper bound for $\EE[\DevT]$, we use the fact that all the training sets have the same length so that
		$$\EE(\DevT) =\EE\left[\risk\left[\alg(T_1)\right]-\ER\left[\alg(T_1),T_1\right]\right]. $$
		Then, we use Lemma 7 from \citet{bousquet2002stability} ensuring that 
		\begin{equation*}
		\forall T\subset [n] \quad , \quad	\EE\left[\risk\left[\alg(T)\right]-\ER\left[\alg(T),T\right]\right]= \EE\left[\loss\left(\alg(T),O'\right)-\loss\left(\alg(T^l),O'\right)\right].
		\end{equation*}
		Where  $\alg(T^{l})$ is the learning rule $\alg$ trained on the sample $\DD_{T}\setminus \{O_l\}\cup \{O'\}$.
		Replacing the left side term by $\EE[\DevT]$ and using Definition \ref{def:unif-stable} gives
		\begin{equation*}
		\abs{\EE[ f(\mathcal D) ]}  = \abs{\EE[\DevT]} \leq 2\beta_{n_{T}}.
		\end{equation*}
Since $ |f(\mathcal D) | \leq  | Z| + 2\beta_{n_{T}} $, we simply use \eqref{ub_proba_useful} to reach the conclusion. 
	\end{proof}
\begin{remark}\label{remark:bousquet-UB-gen}
	Applying Lemma \ref{lemma:Dev-train-UB} with $T=[n]$ and $K=1$ gives the following probability upper bound
		$$\PP\left(\big|\ER\big[\alg([n]),[n]\big]-\risk\big[\alg([n])\big]\big| \geq t + 2\beta_n \right) \leq 2\exp{\left(\ffrac{-2nt^2}{(2L+4\beta_{n}n)^2} \right)}.$$
	Thus we retrieve the bound of Theorem 12 in \citet{bousquet2002stability}.
\end{remark}

\section{DETAILED PROOFS}\label{sec:detailedProofs}

\subsection{Proof of Theorem \ref{theo:CV-stable-bound}}\label{proof:CV-UB}
We proceed as described in the sketch of proof. Using Equation \ref{ineq:error-decomp}, write 
\begin{align*}
\left|\rCV\left[\alg,V_{1:K}\right]-\TR\big[\alg([n])\big]\right| \leq \abs{\DevCV}+\abs{\BiasCV}.
\end{align*}	
It remains to combine Lemma~\ref{lemma:Dev-CV-UB} with Fact~\ref{fact:Bias-UB} to obtain the desired result.
\subsection{Proof of Theorem \ref{theo:K-fold-LB}}\label{sec:proof-LB-RERM}

	Consider the  regression problem  for a random pair  $O=(X,Y)$ consisting of a  covariate  $X=\ffrac{\epsilon}{M}$ where $\epsilon$ is a Rademacher variable and  a response $Y=M\sign{X}$ for some $M>1$. Namely $\epsilon\in\left\{+1,-1\right\}$ and $P(\epsilon=\pm 1)=1/2$.\\
	
	Now, let $n\leq e^M$  and define the algorithm $\alg$ as a regularized empirical risk minimizer, more precisely $\alg(\DD,X)=\hat \beta\left(\DD\right) X$ with
	
	$$\hat \beta\left(\DD\right)=\argmin_{\beta\in \rset}\ffrac{1}{n}\sum_{i=1}^{n}(Y_i-\beta X_i)^2+\lambda_n\lvert\beta\rvert^2, $$
	
	and $\lambda_n=\ffrac{1}{\log\left(n\right)}-\ffrac{1}{M^2}$. 
	
	\paragraph{Algorithmic Stability} It's easy to check that 
	$\hat \beta=\ffrac{\overline{X_nY_n}}{\overline{(X_n)^2}+\lambda_n
	}$ where $\overline{X_nY_n}=\frac{1}{n}\sum_{i=1}^{n}X_iY_i$. Moreover, using the fact that $X_iY_i=1$, one obtains
	
	$$\hat \beta\left(\DD\right)=\ffrac{1}{1/M^2+\lambda_n}=\log(n).$$
	
	On the other hand, write
	\begin{align*}
	\loss\big(\alg\left(\DD\right),O\big)-\loss\big(\alg(\DD^{\backslash i}),O\big)&=(\beta_{n-1}-\beta_{n})X\left(2Y-(\beta_n+\beta_{n-1})X\right)\\
	&=\left(\beta_{n-1}-\beta_{n}\right)\left(2-\ffrac{(\beta_n+\beta_{n-1})}{M^2}\right),\\
	&=\left(\log(n-1)-\log(n)\right)\left(2-\ffrac{\log(n)+\log(n-1)}{M^2}\right),
	\end{align*}
	where the second line follows from the fact that $XY=1$ and the last  follows by replacing $\beta$ and $\lambda$ by their expression.\\
	To conclude this part, we use the fact that $\log(1+x)\leq x$ for all $x\leq 1$, to obtain
	$$\big\lvert\loss\big(\alg\left(\DD\right),O\big)-\loss\big(\alg(\DD^{\backslash i}),O\big)\big\rvert \leq \ffrac{2}{n}. $$
	\paragraph{Bias Lower Bound}
	Using the same equation as before we have

	\begin{align*}
	\loss\big(\alg\left(\DD\right),O\big)-\loss\big(\alg(\DD_T),O\big)&=\left(\log(n_T)-\log(n)\right)\left(2-\ffrac{\log(n)+\log(n_T)}{M^2}\right)\\
	&=\log(\frac{K-1}{K})\left(2-\ffrac{\log(n)+\log(n_T)}{M^2}\right).
	\end{align*}	
	
	Thus, since $n_T\leq n \leq e^M$ we obtain 
	
	$$  \risk\left[\alg\left(\DD_T \right)\right]-\risk\left[\alg\left(\DD\right)\right]\geq 2\log(\frac{K}{K-1})\left(1-\ffrac{1}{M}\right).$$

It remains to notice that
\begin{align*}
	\EE\left[\big|\rCV\left[\alg,V_{1:K}\right]-\risk\left[\alg[n]\right] \big|\right] &
	\geq\left|\EE\left[ \rCV\left[\alg,V_{1:K}\right]-\risk\left[\alg[n]\right] \right] \right|\\
	&=\risk\left[\alg\left(\DD_T \right)\right]-\risk\left[\alg\left(\DD\right)\right],
\end{align*}

and the proof is complete.
\subsection{Proof of Theorem \ref{theo:CVC-UB-det}}\label{proof:CVC-UB}
We proceed as described in the sketch of  proof. First remind the error decomposition \ref{ineq:error-decom-cvc}
$$ \rCVC\left[\alg,V_{1:K}\right]-\risk\left[\alg([n])\right] = \DevAll + \ffrac{n_T}{n}(\DevCV-\DevT), $$ 
where $\DevCV$, $\DevT$ and $\DevAll$ are defined in \ref{def:dev_CV}, \ref{def:dev_train} and \ref{def:dev_all} respectively. Since $n_T\leq n$, using the triangular inequality yields
$$ \big|\rCVC\left[\alg,V_{1:K}\right]-\risk\left[\alg[n]\right] \big| \leq \abs{\DevAll} + \abs{\DevCV} +\abs{\DevT} .$$
Combining lemma \ref{lemma:Dev-CV-UB} and \ref{lemma:Dev-train-UB} regarding $\DevCV$ and $\DevT$ with Remark \ref{remark:bousquet-UB-gen} regarding $\DevAll$, one obtains 
\begin{align}
\notag	&\PP\left(\big|\rCVC\left[\alg,V_{1:K}\right]-\risk\left[\alg[n]\right]\big|   \geq t + 2(\beta_{n_{T}}+\beta_{n})\right)  \\
\notag	& \leq  \PP\left(   \abs{\DevCV}    \geq t/3 \right) + 
\PP\left(  \abs{\DevAll}     \geq t/3 + 2 \beta_{n}\right) +
\PP\left(   \abs{\DevT}     \geq t/3 + 2\beta_{n_{T}} \right)\\
\label{eq:CVC-UB-expon-form}	&\leq 6\exp{\left(\ffrac{-2nt ^2 }{ 9(4\beta_{n_{T}}n_T+L)^2} \right)}.
\end{align}
By inverting, and using the assumption $\beta_t\leq \ffrac{\lambda}{t}$ one gets, with probability $1-6\delta$,
$$\abs{\rCVC\left[\alg,V_{1:K}\right]-\risk\left[\alg[n]\right]} \leq 2\lambda(\ffrac{1}{n}+\ffrac{1}{n_T}) +3(4\lambda+L)\sqrt{\ffrac{\log(1/\delta)}{2n}}, $$
which is the desired result.

\subsection{Proof of Theorem \ref{theo:model-selec-CVC}}
The proof of Theorem \ref{theo:model-selec-CVC} relies on the following proposition,
\begin{proposition}\label{prop:uniform-CVC-UB}
	Let $(\alg^{(m)})_{m\in \mathcal{M}}$ be a family of algorithms where each learner $\alg^{(m)}$ is $(\beta_{m,t})_{1\leq i \leq n}$ uniform stable with respect to loss function $0\leq \loss(g,O)\leq L$. Additionally, assume that, $\abs{\mathcal{M}}< \infty$ and that
	$$ \forall m \in \mathcal{M} \:;\: \beta_{m,t} \leq \ffrac{M}{t},$$
	for some $M>0$. Then one has, with probability at least $1-6\delta$,
	\begin{equation*}
		\sup_{m\in \mathcal{M}}\bigg|\ER^{corr}_{\Kfold}\left[\alg^{(m)},V_{1:K}\right]-\TR\big[\alg^{(m)}([n])\big]\bigg|\leq  \ffrac{ 6M }{n} +4(M+L)\sqrt{\ffrac{\log(\abs{\mathcal{M}}/\delta)}{n}}.
	\end{equation*}
\end{proposition}

\begin{proof}
	Using Fact \ref{fact:mask-property} for the \Kfold\:scheme ($n_T=n (K-1)/ K $ with $K\geq 2$), and the fact that $\beta_{m,n}  + \beta_{m,n_T} \leq (M/n)  ( ( 2K - 1) / (K-1)) $ as well as $(2K-1)/(K-1) \leq 3$, one gets
	$$ \forall m\in \mathcal{M}\:,\:	\PP\left(\big|\ER^{corr}_{\Kfold}\left[\alg^{(m)},V_{1:K}\right]-\TR\big[\alg^{(m)}([n])\big]\big|  \geq t + (6M/n) \right)
	\leq 6\exp{\left(\ffrac{-2nt^2}{9(4M+L)^2} \right)},$$
	which gives by a union bound
	$$ 	\PP\left(\sup_{m\in \mathcal{M}}\big|\ER^{corr}_{\Kfold}\left[\alg^{(m)},V_{1:K}\right]-\TR\big[\alg^{(m)}([n])\big]\big|  \geq t+  (6M/n)\right)
	\leq 6\abs{\mathcal{M}}\exp{\left(\ffrac{-2nt^2}{9(4M+L)^2} \right)}.$$
	Thus, by inverting, we obtain the desired result.
\end{proof}
\paragraph{Proof of Theorem \ref{theo:model-selec-CVC}} First, using the definition of $\hat m $ (eq. \ref{def:opt-Kfold}), write
$$\ER^{corr}_{\mathrm{Kfold}}\left[ \alg^{(\hat m)},V_{1:K}\right]-\ER^{corr}_{\mathrm{Kfold}}\left[ \alg^{(m^*)},V_{1:K}\right] \leq 0.$$
It follows that
\begin{align}\label{ineq:key-ineq-mod-selec}
	\risk\big[\alg^{(\hat m)}([n])\big]-\TR\big[\alg^{(m^*)}([n])\big]&\leq   
	\risk\big[\alg^{(\hat m)}([n])\big]-\ER^{corr}_{\mathrm{Kfold}}\left[ \alg^{(\hat m)},V_{1:K}\right]
	+\ER^{corr}_{\mathrm{Kfold}}\left[ \alg^{(m^*)},V_{1:K}\right] -\TR\big[\alg^{(m^*)}([n])\big]\\
	&\leq 2\sup_{m\in \mathcal{M}}\bigg|\ER^{corr}_{\Kfold}\left[\alg^{(m)},V_{1:K}\right]-\TR\big[\alg^{(m)}([n])\big]\bigg|.\nonumber
\end{align}
It remains to use proposition \ref{prop:uniform-CVC-UB} and the proof is complete. 

\section{UNIFORM STABILITY FOR RANDOMIZED ALGORITHMS}\label{sec:randomized}
In this section we generalize the results from he main paper to the case of randomized algorithms. 
Let us  start with reminding the concept of uniform stability for randomized learning algorithms introduced in \citet{elisseeff05a}. 
\begin{definition}
	An algorithm $\alg$ is said to be   $(\beta_t)_{1\leq t\leq n}$-\emph{uniform} stable with respect to a loss function $\loss$ if, for any $\DD\in\mathcal{Z}^n$, $T\subset [n]$, $i\in T$ and $o\in \mathcal{Z}$, the following holds
	\begin{equation}
	\left|\EE_{\alg}\left[\loss\left(\alg(T),o\right)\right]-\EE_{\alg}\left[
	\loss\left(\alg(T^{\backslash i}),o\right)\right]\right|  \leq \beta_{{n_T}}. 
	\end{equation}
	Where the randomness in the latter expectation stems from the algorithm $\alg$ while the observation $o$ and the data sequence $\DD$ are fixed. Equivalently,
		\begin{equation}
\left|\EE\left[\loss\left(\alg(T),O\right)-
	\loss\left(\alg(T^{\backslash i}),O\right)\bigg\lvert \DD_T,O\right]\right| \leq \beta_{{n_T}}. 
	\end{equation}
\end{definition}

Note that a similar version of Fact \ref{fact:Bias-UB} still holds, more precisely, one has

\begin{fact}\label{Bias-UB-random}
	Let $\alg$ be a decision rule which is $(\beta_t)_{1\leq t  \leq n}$ uniformly stable, additionally suppose that the sequence $(\beta_t)_{1\leq t  \leq n}$ is decreasing, then for any $T\subset [n]$,  and $o\in \mathcal{Z}$, one has
	\[\left| \EE\left[\loss\big(\alg([n]),o\big)-
		\loss\big(\alg(T),o\big)\mid \DD\right]\right| \leq \sum_{i=n_T}^{n}\beta_{i} . 
	\]
\end{fact}

\subsection{Upper Bounds for \textrm{K-fold} \CV under Random Uniform Stability }
We now  derive upper bounds -\emph{in expectation}- on the error induced by $\rCV$ and $\rCVC$. As highlighted in Remark \ref{remark:bias-variance-stability}, the main problem with \Kfold\ \CV is it's bias. Therefore, for the sake of brevity, we will focus only on the expectation of the estimate .  
\begin{theorem}\label{theo:randomized}
	Suppose that $\alg$   is $(\beta_t)_{1\leq t  \leq n}$ uniformly stable. Then  we have
	\[ \left|	\EE\left[\rCV\left[\alg,V_{1:K}\right]-\TR\big[\alg([n])\big]\right]\right| \leq  \sum_{i=n_T}^{n}\beta_{i} . \]
\end{theorem}
\begin{proof}
	First, by applying the tower rule one has 
	\begin{align}\label{eq:CV-expectation-randomized}
		\EE\left[\rCV\left[\alg,V_{1:K}\right]-\TR\big[\alg([n])\big]\right]&=\EE\left[ \frac{1}{ K} \sum_{j= 1}^{K}\ER\big[\alg(T_j) , V_j\big]-\TR\left[\alg([n])\right]\ \right] \nonumber\\
		&=\frac{1}{ K} \sum_{j= 1}^{K}\EE\left[ \EE\left[\ER\big[\alg(T_j) , V_j\big]\mid \DD_{T_j}\right]-\TR\big[\alg([n])\big]\ \right] \nonumber\\
		&=\frac{1}{ K} \sum_{j= 1}^{K}\EE\left[ \EE\left[\loss(\alg(T_j,O)\mid \DD_{T_j}\right]-\TR\big[\alg([n])\big]\ \right]\nonumber\\
		&=\frac{1}{ K} \sum_{j= 1}^{K}\EE\left[ \EE\left[\loss(\alg(T_j,O)-\loss(\alg([n],O))\mid \DD \right] \right].
	\end{align}		
	The third line follows from the fact that $\loss(\alg(T),O_j)$ and $\loss(\alg(T),O)$ has the same law for all $j \in V$. This indeed verified since all the training sets $T_j$'s has the same length and the $O_j$'s are independent from $\DD_T$. To obtain the desired result, it remains to combine Equation \ref{eq:CV-expectation-randomized} with fact \ref{Bias-UB-random}.
\end{proof}

Now let us  prove that the bias corrected \Kfold\:(cf. Eq \eqref{def:risk-CV-corr}) has a vanishing bias for randomized algorithms.
\begin{theorem}[Corrected \Kfold\;bias]
	Suppose that $\alg$   is $(\beta_t)_{1\leq t  \leq n}$ uniformly stable and $\beta_t\leq \ffrac{\lambda}{t}$, for some $\lambda>0$. Then, we have
	\[ \left|	\EE\left[\ER_{\mathrm{Kfold}}^{corr}\left[\alg,V_{1:K}\right]-\TR\big[\alg([n])\big]\right]\right| \leq \ffrac{2\lambda(2K-1)}{(K-1)n}. \]
\end{theorem}
\begin{proof}
	Combining the error decomposition \eqref{ineq:error-decom-cvc} with the fact that  $n_T=\frac{K-1}{K}n$ we obtain 
	\begin{equation}\label{eq:CVC-decomp-Kfold}
	\ER_{\mathrm{Kfold}}^{corr}\left[\alg,V_{1:K}\right]-\risk\left[\alg[n]\right] = \DevAll+ \ffrac{K-1}{K}(\DevCV-\DevT).
	\end{equation}
	
	Using  Theorem 2.2 in \citet{hardt16}, one obtains the twin inequality
	\[ \left|\EE\left[\DevAll\right]\right| \leq 2\beta_n \]
	\[ \left|\EE\left[\DevT\right]\right| \leq 2\beta_{n_T}.  \]	
	Since $\EE\left[\DevCV\right]=0$, Equation \ref{eq:CVC-decomp-Kfold} combined with the triangular inequality gives
	\begin{align}\label{ineq:random-Kfold}
		\left|	\EE\left[\ER_{\mathrm{Kfold}}^{corr}\left[\alg,V_{1:K}\right]-\TR\big[\alg([n])\big]\right]\right| &\leq 2\left(\beta_n+\ffrac{(K-1)\beta_{n_T}}{K}\right) \nonumber\\
		&\leq 2\left(\beta_n+\beta_{n_T}\right).
	\end{align}
	It remains to use the assumption $\beta_t\leq \ffrac{\lambda}{t}$ and the proof is complete.	
\end{proof}
We conclude this section by deriving an upper bound for the model selection problem,
\begin{theorem}	Let $(\alg^{(m)})_{m\in \mathcal{M}}$ be a family of algorithms where each learner $\alg^{(m)}$ is $(\beta_{m,t})_{1\leq i \leq n}$ uniform stable with respect to loss function $\loss$. Additionally, assume that
	$$ \beta_{m,t} \leq \ffrac{M}{t} $$
	for some $M>0$.  Then one has 
	\[\EE\left[\risk\big[\alg^{(\hat m)}([n])\big]-\TR\big[\alg^{(m^*)}([n])\big]\right] \leq \ffrac{4M(2K-1)}{(K-1)n}, \]
	where $m^*$ and $\hat m $ are defined by Equations \ref{def:oracle}, \ref{def:opt-Kfold} respectively.
\end{theorem}
\begin{proof}
	First, using Equation \ref{ineq:random-Kfold}, one obtains, for all $m\in \mathcal{M}$,
	\begin{align*}
		\left|\EE\left[\ER_{\mathrm{Kfold}}^{corr}\big[\alg^{(m)},V_{1:K}\big]-\TR\big[\alg^{(m)}([n])\big]\right]\right| &\leq 2\left(\beta_{m,n}+\beta_{m,n_T}\right)\\
		&\leq \ffrac{2M(2K-1)}{(K-1)n}.
	\end{align*}
	So that
	\begin{equation}\label{ineq:exp-model-selec-UB}
	\sup_{m\in\mathcal{M}}\left|\EE\left[\ER_{\mathrm{Kfold}}^{corr}\left[\alg^{(m)},V_{1:K}\right]-\TR\big[\alg^{(m)}([n])\big]\right]\right|\leq \ffrac{2M(2K-1)}{(K-1)n}. 
	\end{equation}
	On the other hand, Inequality \ref{ineq:key-ineq-mod-selec} yields
	\begin{align*}
		\EE\left[\risk\big[\alg^{(\hat m)}([n])\big]-\TR\big[\alg^{(m^*)}([n])\big] \right]&\leq  
		\EE\left[\risk\big[\algdd^{(\hat m)}([n])\big]-\ER^{corr}_{\mathrm{Kfold}}\big[ \alg^{(\hat m)},V_{1:K}\big]\right]\\&\hspace{8mm}
		+\EE\left[\ER^{corr}_{\mathrm{Kfold}}\big[ \alg^{(m^*)},V_{1:K}\big] -\TR\big[\alg^{(m^*)}([n])\big]\right]\\
		&\leq 2\sup_{m\in\mathcal{M}}\left|\EE\left[\ER_{\mathrm{Kfold}}^{corr}\big[\alg^{(m)},V_{1:K}\big]-\TR\big[\alg^{(m)}([n])\big]\right]\right|.
	\end{align*}
	Thus, by Inequality \ref{ineq:exp-model-selec-UB}, one has
	$$	\EE\left[\risk\big[\alg^{(\hat m)}([n])\big]-\TR\big[\alg^{(m^*)}([n])\big] \right] \leq \ffrac{4M(2K-1)}{(K-1)n},$$
	which concludes the proof.
\end{proof}
\subsection{Proof of Theorem \ref{theo:K-fold-LB-SGD}}\label{sec:proof-LB-SGD}
Following the line of \cite{zhang2022stability}, consider the following convex function  ,

$$ f(w,o)=\ffrac{1}{2}w^\top A w - yx^\top w, $$

where  $A$ is  a positive semi definite penalization matrix (PSD) in $\rset^{d\times d}$ with rank $p<d$, $w\in \rset^d$ and $x,y\in \rset^d\times\rset$ . Such a rank deficient penalization matrix can be found in multiple contexts like fused lasso \citep{tibshirani2011solution}, fused ridge \citep{bilgrau2020targeted}, etc.\\
In the next lemma, by picking carefully the input space $\mathcal{X}\times\mathcal{Y}$ and the distribution P, we construct an example where we control exactly the amount of \emph{instability} of SGD .
Theorem \ref{theo:K-fold-LB-SGD} follows directly from the following proposition.
\begin{proposition}
	Let $M>1$ , $n\in \llbracket 1,e^M \rrbracket $ . Suppose that $O=(X,Y)\in\mathcal{X}\times\mathcal{Y}=\left\{v,-v\right\}\times\left\{1\right\}$  and $P(X= v)=\frac{2}{3}$ where $v$ is a unit vector in $\rset^d$ such as $Av=0$ . For $t\geq 1$, there exist a sequence of step sizes $(\alpha_{k,n})_{1\leq k \leq t}$  , such as the SGD algorithm (Definition \ref{def:SGD}) with  $\loss=f$ and $\alg_0=0$ verifies,
	
	$$\EE\left[\big|\rCV\left[\alg,V_{1:K}\right]-\risk\left[\alg[n]\right] \big|\right]\geq \frac{\log\left(K/K-1\right)}{3}.$$
	
	Furthermore SGD satisfies Assumption \ref{assum:stability-setting}  with respect to $\loss$ with $L=M$ and 
	$$\beta_n \leq \ffrac{3}{n-1}\sum_{k=1}^{t}\alpha_{k,n}\leq \ffrac{3M}{n-1}.$$ 
\end{proposition}

\begin{proof}
	Computing the gradient of $f(\cdot,o)$ for $o=(x,y)$ yields,
	$$\nabla_wf(w,o)=Aw-yx.$$
	Now set $w_t=\alg_t\left(\left[n\right]\right)$ and write using the definition of SGD (Equation \ref{def:SGD}) :
	\begin{equation}\label{eq:update-rule}
	w_{t+1}=\left(I-\alpha_tA\right)w_t+\alpha_{t,n}y X_{it}.
	\end{equation}
	Thus by induction we obtain  $w_t=\theta_t v$ for some $\theta_t \in \rset$. Consequently, Equation \ref{eq:update-rule} yields,	
	\begin{equation*}
	w_{t+1}=w_t+\alpha_{t,n}y X_{it},
	\end{equation*}
	so that,
	\begin{equation}\label{eq:update-rule-complete}
		w_{t}=\sum_{k=1}^{t}\alpha_{k,n}y X_{it}.
	\end{equation}	
	Furthermore SGD picks $X_{it}=v$ with probability $n_+/n$ where $n_+$ (\emph{resp.} $n_-$) is the number of samples such as $X=v$ (\emph{resp.} $-v$), hence
	\begin{equation}\label{eq:SGD-path-exact}
	\EE_\alg\left[w_{t}\right]=\sum_{k=1}^t\alpha_{k,n}\left(\frac{n_+}{n}-\frac{n_-}{n}\right)v.
	\end{equation}
	On the other hand, since $w_t=\theta_tv$ we obtain,
	\begin{equation}\label{eq:SGD-loss-exact}
	f(w_t,o)=-yx^\top w_t.
	\end{equation}
	Set $w^{\prime}_t=\alg\left(\left[n\right]^{\backslash j}\right)$ and consider the case where $j$ is such as $X_j=v$. Note that the other case is similar and thus omitted. Since $\lVert v \rVert =1$, one  has  by Equations \ref{eq:SGD-path-exact} and \ref{eq:SGD-loss-exact}
	$$ \forall o \in \mathcal{Z} \:,\: \left|\EE_\alg\left[f\left(w^{\prime}_{t},o\right)\right]-\EE_{\alg}f\left[\left(w_{t},o\right)\right]\right|=\left|\frac{n_+-n_-}{n}\sum_{k=1}^t\alpha_{k,n}-\frac{n_+-n_- -1}{n-1}\sum_{k=1}^{t}\alpha_{k,n-1}\right|. $$
	Now for $t\geq 0$ , take $\alpha_{k,n}=\ffrac{\log(n)}{t}$ for all $k\leq t$ so that,
	$$ \forall o \in \mathcal{Z} \:,\: \left|\EE_\alg\left[f\left(w^{\prime}_{t},o\right)\right]-\EE_{\alg}f\left[\left(w_{t},o\right)\right]\right|=\left|\frac{n_+-n_-}{n}\log(n)-\frac{n_+-n_- -1}{n-1}\log(n-1)\right|, $$
	which yields by simple algebra
	\begin{align}\label{eq:SGD-stability-ineq}
		\forall o \in \mathcal{Z} \:,\:\left|\EE_\alg\left[f\left(w^{\prime}_{t},o\right)\right]-\EE_{\alg}f\left[\left(w_{t},o\right)\right]\right|&\leq \ffrac{2\log\left(n\right)+1}{n-1}\nonumber\\&\leq \ffrac{3\log(n)}{n-1}\nonumber\\&=\ffrac{3}{n-1}\sum_{k=1}^{t}\alpha_{k,n}\nonumber\\&\leq \ffrac{3M}{n-1}.
	\end{align} 
	Now, using Equations \ref{eq:update-rule-complete} and \ref{eq:SGD-loss-exact} and the expression of $\alpha_{k,n}$ we get,
   $$ \forall o \in \mathcal{Z} \:,\:\lvert f(w_t,o)\rvert=\sum_{k=1}^{t}\alpha_{k,n}=\log(n)\leq M,$$	
	The latter equation combined with Equation \ref{eq:SGD-stability-ineq} confirms that SGD  verifies Assumption \ref{assum:stability-setting}with $L=M$ and 
	$$\beta_n \leq \ffrac{3}{n-1}\sum_{k=1}^{t}\alpha_{k,n}\leq \ffrac{3M}{n-1}.$$  
	For the lower bound, let $T\subset  \left[n\right]$ and set $w_t^{T}=\alg_t\left(T\right)$. Using \ref{eq:SGD-path-exact} yields
	
	$$ 	\EE\left[\EE_\alg\left[w_{t}^T\right]\right]=\sum_{k=1}^t\alpha_{k,n_T}\frac{v}{3}, $$
	so that by \ref{eq:SGD-loss-exact} ,
	\begin{align*}
	\left|\EE\left[\EE_\alg\left[f\left(w_{t},o\right)\right]\right]-\left[\EE_\alg\left[f\left(w^{T}_{t},o\right)\right]\right]\right|&=\frac{\sum_{k=1}^t\alpha_{k,n}-\sum_{k=1}^t\alpha_{k,n_T}}{3}\\
	\left(\alpha_{k,n}=\log(n)/t\right)	&=\frac{\log\left(K/K-1\right)}{3}.
	\end{align*} 	
	It remains to notice that
	\begin{align*}
	\EE\left[\big|\rCV\left[\alg,V_{1:K}\right]-\risk\left[\alg[n]\right] \big|\right] &
	\geq\left|\EE\left[ \rCV\left[\alg,V_{1:K}\right]-\risk\left[\alg[n]\right] \right] \right|\\
	&=\left|\risk\left[\alg\left(\DD_T \right)\right]-\risk\left[\alg\left(\DD\right)\right]\right|,
	\end{align*}
	and the proof is complete.
\end{proof}


\end{document}